\documentclass{amsart}
 
\usepackage{colordvi,multicol,color}
\usepackage{amssymb}
\usepackage[hyperindex,bookmarksnumbered,plainpages]{hyperref}

\usepackage{enumerate}
\usepackage{multicol}
\newtheorem{theorem}{Theorem}[section]
\newtheorem{lemma}[theorem]{Lemma}
\newtheorem{corollary}[theorem]{Corollary}
\theoremstyle{definition}
\newtheorem{definition}[theorem]{Definition}

\newtheorem{proposition}[theorem]{Proposition}

\newtheorem{remark}[theorem]{Remark}

\numberwithin{equation}{section}

\DeclareMathOperator{\ad}{ad}
\DeclareMathOperator{\field}{\mathbb{F}}
\DeclareMathOperator{\g}{\mathfrak{g}}

\DeclareMathOperator{\vspan}{span}

\begin{document}

\title{Lie $2$-algebras of toral rank $3$}

\author[Benitez]{Germ\'an Benitez}
\address{Universidade Federal do Paran\'a \\ Pontal do Paran\'a \\ Brazil \& Departamento de Matem\'atica \\ Universidade Federal do Amazonas \\ Manaus \\ Brazil}
\email{gabm@ufam.edu.br - gabm03@gmail.com}

\author[Payares Guevara]{Carlos R. Payares Guevara}
\address{Dirección de Ciencias B\'asicas \\ Universidad Tecnol\'ogica de Bol\'ivar \\ Cartagena de Indias \\ Colombia \&  Departamento de Matem\'atica \\ Universidade Federal do Amazonas \\ Manaus \\ Brazil}
\email{cpayares@utb.edu.co}

\author[Quintero Vanegas]{Elkin O. Quintero Vanegas}
\address{Departamento de Matem\'atica \\ Universidade Federal do Amazonas \\ Manaus \\ Brazil}
\email{eoquinterov@ufam.edu.br}



\date{\today}



\begin{abstract}
In this paper we study Lie $2$-algebras over an algebraically closed field of characteristic two, which have a triangulable Cartan subalgebra, and derive some general properties of centerless ones. These properties allow us to do an analysis on simple Lie $2$-algebras of toral rank three and provide a necessary condition for simplicity. By means of this latter condition we also conclude that simple Lie $2$-algebras with a triangulable Cartan subalgebra of toral rank three and of  dimension less than or equal to $16$ cannot exist.
\end{abstract}

\subjclass[2010]{17B50, 17B20, 17B22}

\keywords{Restricted simple Lie algebra, toral rank, root space decomposition}

\maketitle




\section*{Introduction}

Simple Lie algebras over an algebraically closed field of characteristic zero were classified by W. Killing (1888) and E. Cartan (1894). They have found four infinite families $A_n, B_n, C_n, D_n$ and five exceptional cases $E_6, E_7, E_8, G_2$ and $F_4$. The study of finite-dimensional simple Lie algebras over algebraically closed fields of positive characteristic have been initiated by N. Jacobson, E. Witt and H. Zassenhaus in the 1930s. By use of Chevalley basis it is possible to choose the structure constants to be integers, and  by modulo $p$ reduction some modular Lie algebras over a field of characteristic $p$ are obtained. The resulting algebras are called classical Lie algebras and all of them are simple for any prime $p$ greater than five~\cite[Ch. 4]{Strad04}. In addition to classical simple Lie algebras there are also known four infinite families of non-classical simple Lie algebras. They are of graded Cartan type and fall into the following  four classes: Witt $\bf W$, Special $\bf S$, Hamiltonian $\bf H$ and Contact $\bf K$ (see~\cite[Ch. 4.2]{Strad04}). Surprisingly, the characteristic five case is known to accommodate a special type of simple Lie algebras that only appear in this case, namely the Melikian type algebras~\cite{Mel}.

To achieve the classification of the finite-dimensional simple Lie algebras over an algebraically closed field of characteristic $p>3$ it proved useful to first classify  a subclass of simple Lie algebras known as restricted Lie algebras or Lie $p$-algebras introduced by N. Jacobson \cite[p. 187]{ja} for arbitrary characteristic $p$. More precisely, A. Kostrikin and I. Shafarevich conjectured in 1966 \cite{KosShaf66} that: 

\vspace{.2cm}

\emph{Every finite-dimensional simple Lie $p$-algebra over a field of characteristic greater than five is either of  classical type or of Cartan type.} 

\vspace{.2cm}

R. Wilson in \cite{Wil77} proved that all Cartan subalgebras $\mathfrak{h}$ of a finite-dimension simple Lie algebra $\mathfrak{g}$ over an algebraically closed field of characteristic $p>7$ are triangulable, i.e., $[\mathfrak{h}, \mathfrak{h}]$ acts nilpotently on $\mathfrak{g}$. This deep result was the starting point of the classification problem for simple Lie algebras (non necessarily restricted), and allowed to R. Block and R. Wilson to prove in \cite{BlWil88} the latter conjecture for $p>7$. Later, A. Premet in \cite{Pre94} extended the Wilson's triangulable  theorem to the case where $p=7$ and pointed out that it does not hold in characteristic $5$. 

	Several results in \cite{BlWil88} were extended for $p>3$ by A. Premet and H. Strade. Thus, considering also the Melikian family, they completed the classification of simple Lie $p$-algebras for $p>3$. In summary, over an algebraically closed field of characteristic greater than three the following statement holds (see \cite[pp. 180 ff.]{Strad04}):

\vspace{.2cm}

\emph{Every finite-dimensional simple Lie $p$-algebra is of either classical type, Cartan type or Melikian type.} 

\vspace{.2cm}

It deserves to be noticed that in the characteristic five case the only simple Lie $p$-algebra of Melikian type is of dimension $125$.

V. Kac generalized the Kostrikin--Shafarevich conjecture for simple Lie algebras \cite{Kac71,Kac74}. The proof  took few decades and the efforts of few mathematicians, and was brought to completion in series of papers  by H. Strade and A. Premet~\cite{PreStra97,PreStra99,PreStra01,PreStra04, PreStra07,PreStra08}. The Block--Wilson--Strade--Premet classification theorem is a landmark result of modern mathematics and it reads:

\vspace{.2cm}

\textbf{Theorem} (Block--Wilson--Strade--Premet). \emph{Every finite-dimensional simple Lie algebra over an algebraically closed field of characteristic greater than three is of either classical type, Cartan type or Melikian type.}

\vspace{.2cm}

The classification of finite-dimensional simple Lie algebras over a field of characteristic both two and three remains open. In contrast with the case of characteristic greater or equal to five, new phenomena in characteristic two and three are already known in the literature. For instance, the Lie algebras of Cartan type are not necessarily simple over characteristic two. The first steps towards the classification of simple Lie algebras in low characteristics were taken by S. Skryabin~\cite{Sk} by means of proving that every finite-dimensional Lie algebra over an algebraically closed field of characteristic two and absolute toral rank one is solvable, i.e., the absolute toral rank of a finite-dimensional simple Lie algebra is at least two.

In~\cite[p. 210]{PreStra06}, A. Premet and H. Strade presented the following problem: 

\vspace{.2cm}

\emph{Classify all the finite-dimensional simple Lie algebras of (absolute) toral rank two over an algebraically closed field of characteristic two and three.} 

\vspace{.2cm}

A. Grishkov and A. Premet have taken some steps in this direction (see \cite{GrishPrem} and \cite[Problem 1]{PreStra06}). They announced that all finite-dimensional simple Lie algebras over an algebraically closed field of characteristic two and absolute toral rank two are the classic ones. In particular, all finite-dimensional simple Lie $2$-algebras of (relative) toral rank two have dimension $8$, $14$ or $26$.

	To the best of our knowledge there are no results in the literature regarding simple Lie algebras with absolute toral rank greater than two. Thus, we are naturally interested in the following problem:

\vspace{.2cm}

\emph{Classify the finite-dimensional simple Lie algebras of absolute toral rank greater than two over an algebraically closed field of characteristic two.}

\vspace{.2cm}

	The latter problem is very general and at the time of writing we are unaware how to attack it. However, with the Kostrikin--Shafarevich conjecture and its resolution in mind, one can easily formulate a weaker problem that might  help understand better the structure of the Lie algebras of toral rank greater than two in characteristic two, namely:

\vspace{.2cm}

{\bf Problem:} \emph{Classify the finite-dimensional simple Lie $2$-algebras of (relative) toral rank greater than two.}

\vspace{.2cm}

	It is worth pointing out that, following the same spirit for $p>5$ and suggested by the simple Lie $2$-algebras  Cartan type $\textbf{W}(3,\textbf{1})$, $\textbf{K}(5,\textbf{1})$, $\textbf{S}(4,\textbf{1})$, $\textbf{H}(6,\textbf{1})$ and $\mathcal{H}(8,\textbf{1})$, we will consider that, all Lie $2$-algebras in this paper have a triangulable Cartan subalgebra. Note that for the algebras $\textbf{K}(5,\textbf{1})$ and $\mathcal{H}(8,\textbf{1})$, we refer the algebra introduced in~\cite{Zhang92} and the subalgebra of $\textbf{H}(8,\textbf{1})$ introduced in~\cite[p. 138]{Purs18}, respectively.

Very recently, C. Payares Guevara and F. Arias Amaya pondering on the weaker problem above proved that classical simple Lie $2$-algebras of odd (relative) toral rank do not exist \cite{PayFab19}. The upshot of our paper is that \emph{the $\mathfrak t$-roots spaces of a simple Lie $2$-algebra of toral rank three have the same dimension} (Theorem~\ref{Payarin}). The careful reader might have already observed that an immediate corollary of the aforementioned theorem is that simple Lie $2$-algebras $\mathfrak g$ of toral rank $3$ with $\dim(\mathfrak{g})\leq 16$ do not exist.

This paper is organized as follows. We begin with, in Section~\ref{sec:prel}, by discussing some preliminaries. In Section \ref{sec:MT}, we fix that all Lie $2$-algebras in this paper have a triangulable Cartan subalgebra and we prove, among other things, that any Lie $2$-algebra has at most seven $\mathfrak t$-roots when $\mathfrak t$ is a maximal torus of dimension three. Section~\ref{sec:centerless} is dedicated to centerless Lie $2$-algebras in the particular case when all seven $\mathfrak t$-roots appear and culminates in the fact that whenever there exist two $\mathfrak t$-roots spaces with different dimension it is possible to construct a proper ideal. Finally, we employ these results to obtain a necessary condition for a Lie $2$-algebra to be simple.


\section{Preliminaries}
\label{sec:prel}

In this paper $\mathbb{F}$ denotes an algebraically closed field of characteristic $2$ and all Lie algebras are assumed to be finite-dimensional over $\mathbb{F}$. We denote by $\ad_{\g}$ the adjoint representation of a Lie algebra $\g$ and by $\ad_{\g}^n(x)$ the composition $\left(\ad_{\g}(x)\right)^n$ for all $x\in \g$ and $n\in\mathbb{N}$. We shall be interested in \emph{Lie $2$-algebras} over $\field$.


\begin{definition} 
Let $\mathfrak{g}$ be a Lie algebra. A map $[2]: \mathfrak{g}\longrightarrow \mathfrak{g}$ such that $a\mapsto a^{[2]}$ is called \emph{$2$-map} if
\begin{enumerate}[(i)]
\item $(\lambda\,a)^{[2]}=\lambda^{2}\,a^{[2]}$, for all $\lambda\in \mathbb{F}$ and for all $a\in\mathfrak{g}$.
\item $\ad_{\mathfrak{g}}^{2}(a)=\ad_{\mathfrak{g}}\left(a^{[2]}\right)$, for all $a\in\mathfrak{g}$.
\item $(a+b)^{[2]}=a^{[2]}+b^{[2]}+[a,b]$, for all $a,b\in\mathfrak{g}$. 
\end{enumerate}
It is important to observe that not all Lie algebras admit a 2-map. For this reason, a Lie algebra with a 2-map is called  \emph{Lie $2$-algebra} and will be denoted by $(\mathfrak{g},[2])$. 
\end{definition}


	Albeit the definitions in this section are valid for arbitrary characteristic $p$, we are interested only in characteristic $p=2$. Notice also that a $2$-map on a centerless Lie $2$-algebra $\mathfrak{g}$ is unique. 

	Let $\mathfrak{g}$ be a Lie $2$-algebra. A \emph{$2$-subalgebra} is a subalgebra of $\mathfrak{g}$ which is closed under the $2$-map. A \emph{simple Lie $2$-algebra} is a Lie $2$-algebra which has $\dim(\g)\neq 1$ and does not have nonzero proper ideals. For instance, Cartan subalgebras of $\mathfrak{g}$ are Lie $2$-subalgebras of $\mathfrak{g}$.

	An element $x\in\mathfrak{g}$ is called \emph{semisimple} (respectively, \emph{$2$-nilpotent}) if $x$ lies in the $2$-subalgebra of $\mathfrak{g}$ generated by $x^{[2]}$, that is $x\in \vspan\left\{{x}^{[2]},{x}^{[2]^{2}},\ldots\right\}$
(respectively, if ${x}^{[2]^{n}}=0$ for some  $n\in \mathbb{N}$). A subset $\mathfrak{n}\subseteq\mathfrak{g}$  is called \emph{$2$-nilpotent} if all $x\in \mathfrak{n} $  is $2$-nilpotent. Engel's theorem implies that every $2$-nilpotent element is nilpotent.

	It is well-known that for any $x \in\mathfrak{g}$ there are unique elements  $x_{s}$ and $x_{n}$ in $\mathfrak{g}$ such that $x_{s}$ is semisimple, $x_{n}$ is $2$-nilpotent, and  $x = x_{s} + x_{n}$ with $[x_{s},x_{n}]=0$ (Jordan-Chevalley-Seligman decomposition, see \cite[p. 81]{SF88}).


\begin{definition}
A $2$-subalgebra $\mathfrak{t}$ of $\mathfrak{g}$ is called \emph{torus of} $ \mathfrak{g}$ if the 2-map is invertible on $\mathfrak{t}$.
\end{definition}


	It follows from Theorem $13$ in \cite[pp. 192--193]{ja} that for a torus $\mathfrak{t}$ of a Lie $2$-algebra $\mathfrak{g}$ there is a basis $\left\{t_{1},\dots, t_{n}\right\}$ such that $t^{[2]}_{i}=t_{i}$. The elements satisfying $t=t^{[2]}$ are called \emph{toral elements}. A torus $\mathfrak{t}_{1}$ of $\mathfrak{g}$ is called \emph{maximal} if the inclusion $\mathfrak{t}_{1}\subseteq \mathfrak{t}_{2}$, with $\mathfrak{t}_{2}$ torus of $\mathfrak{g}$,  implies $\mathfrak{t}_{1} =\mathfrak{t}_{2}$. The concept of relative toral rank first appeared in \cite{Str89} and is one of the most definitions in this paper.
 

\begin{definition}
The \emph{relative toral rank} (\emph{toral rank} for short) of a  Lie $2$-algebra $\mathfrak{g}$ is given by 
	$$
	{MT}(\mathfrak{g}):= \max \left\{\dim(\mathfrak{t})\mid \mathfrak{t}\text{ is a torus of }  \mathfrak{g} \right\}. 
	$$
\end{definition}


	By Theorem 4.1 in \cite[p. 86]{SF88} we conclude that the centralizer $\mathfrak{c}_{\mathfrak{g}}(\mathfrak{t})$ of any maximal torus $\mathfrak{t}$ in $\mathfrak{g}$ is a Cartan subalgebra of $\mathfrak{g}$. Let $\mathfrak{t}$ be a maximal torus of $\mathfrak{g}$,  $\mathfrak{h}:=\mathfrak{c}_{\mathfrak{g}}(\mathfrak{t})$, and let $\ad_{\mathfrak{g}}:\mathfrak{g}\longrightarrow \text{End}_{\field}(\mathfrak{g})$ be the adjoint $2$-representation (i.e., $\ad_{\mathfrak{g}}(x^{[2]})=\ad_{\mathfrak{g}}^{2}(x)$ for all $x\in \mathfrak{g}$). Since $\mathfrak{t}$ is abelian, $\ad_{\mathfrak{g}}(\mathfrak{t})$ is abelian and consists of semisimple  elements. Therefore,  $\g$ can be decomposed into weight spaces with respect to $\mathfrak{t}$ as
	$$
	\mathfrak{g}=\bigoplus_{\lambda\in {\mathfrak{t}}^{*}}\mathfrak{g}_{\lambda}, \quad \text{where}\quad \mathfrak{g}_{\lambda}:=\left\{v\in \mathfrak{g}\mid [t,v]= \lambda(t)v,\,\, \mbox{for all }t\in \mathfrak{t}\right\}.
	$$
The set of roots of $\mathfrak{g}$ with respect to $\mathfrak{t}$ denoted by $\Delta:=\left\{\lambda \in\mathfrak{t}^{*}\setminus\left\{0\right\}\mid \mathfrak{g}_{\lambda}\neq 0\right\}\subseteq \mathfrak{t}^{*}$ is called the $\mathfrak{t}$-\textit{roots} of $\mathfrak{g}$. Since $\mathfrak{g}_0=\mathfrak{h}$ we can conclude that 
	\begin{equation}
	\label{decomproot}
	\mathfrak{g}=\mathfrak{h}\oplus\bigoplus_{\lambda\in \Delta}\mathfrak{g}_{\lambda}, \quad \text{where}\quad \mathfrak{g}_{\lambda}=\left\{v\in \mathfrak{g}\mid [t,v]= \lambda(t)v,\,\, \mbox{for all }t\in \mathfrak{t}\right\}
	\end{equation}
is the \emph{root space decomposition of $\mathfrak{g}$ relative to $ \mathfrak{t}$}.


\begin{remark} 
\label{patrondepatrones}
If $t$ is a toral element of $\mathfrak{t}$, then $\lambda(t)\in \{0,1\}$, for any  $\lambda\in\Delta $. In fact,  
	$$
	\lambda(t)v=\left[t,v\right]=\left[t^{[2]},v\right]=\left[t,\left[t,v\right]\right]=\left[t,\lambda(t)v\right]=\lambda(t)\left[t,v\right]=\lambda(t)^2v, \ \mbox{for }v\in \g\setminus \{0\}.
	$$
\end{remark}



\section{A discussion on arbitrary toral rank}
\label{sec:MT}

Henceforth we consider the Lie $2$-algebra $\mathfrak{g}$ as in~\eqref{decomproot}. Jordan--Chevalley--Seligman decomposition (see \cite[p. 81]{SF88}) implies that the Cartan subalgebra has a decomposition $\mathfrak{h}=\mathfrak{t}\oplus \mathfrak{n}$, where $\mathfrak{t}$ is the maximal torus and $\mathfrak{n}$ is the set of $2$-nilpotent elements of $\mathfrak{g}$ contained in $\mathfrak{h}$ and $[\mathfrak{t},\mathfrak{n}]=0$.


\begin{remark}
\label{rem:nilp}

Throughout this paper we will consider that $\mathfrak{t}$ is \emph{standard} in $\mathfrak{g}$, that means, $\mathfrak{n}$ is an ideal of $\mathfrak{h}$. It is equivalent to $\mathfrak{h}$ be \emph{triangulable} on $\mathfrak{g}$, i.e., $[\mathfrak{h},\mathfrak{h}]$ acts nilpotently on $\mathfrak{g}$, for details see \cite{BlWil88}.	Furthermore, write $\mathfrak{g}_{\xi}^{[2]}:=\left\{x^{[2]}\mid x\in\g_{\xi}\right\}$. It follows from \cite[p. 31]{Strad04} that each $\mathfrak{t}$-root $\xi$ can be extended to $\xi\in\mathfrak{h}^*$ with $\xi\big|_{\mathfrak{n}}=0$. Hence, $\g_{\xi}^{[2]}\subseteq \ker(\xi)$ and this condition implies that 
	$$
	\left[\mathfrak{g}_{\xi},\mathfrak{g}_{\xi}\right]\subseteq\vspan\g_{\xi}^{[2]}\subseteq \ker(\xi)\ \ \ \mbox{for all }\ \xi\in\Delta.
	$$
\end{remark}


\begin{proposition}
\label{wefef}
	Let $\mathfrak{g}$ be a centerless Lie $2$-algebra with $MT(\mathfrak{g})=r$. Then there exists a basis of $\mathfrak{t}^{\ast}$ determined by $\mathfrak{t}$-roots of $\g$. Furthermore, $\dim(\mathfrak{g})\geq 2r.$
\end{proposition}


\begin{proof}
	The case $r=1$ is easy to check. Indeed, if the thesis of the proposition was false then the toral element would belong to the center. For $r>1$, suppose for a contradiction that every basis of $\mathfrak{t}^{\ast}$ has at most $r-1$ $\mathfrak{t}$-roots. From Definition 3 and Remark 2.2.6 in \cite{PreStra06} we conclude that ${MT}(\g)\leq \dim \left(\vspan \Delta\right)\leq r-1$, which contradicts our assumption that ${MT}(\g)=r$.

	The second part of the proposition is an immediate consequence of the decomposition \eqref{decomproot} and the fact that $\dim(\mathfrak{h})\geq r$.
	
\end{proof}


\begin{lemma}
\label{lem:dim=1}
Let $\mathfrak{g}$ be a centerless Lie $2$-algebra. If $\dim\left(\mathfrak{g}_\xi\right)=1$ for all $\xi\in\Delta$, then 
	$$
	\mathfrak{I}:=\mathfrak{n} \oplus\bigoplus_{\xi\in\Delta} \mathfrak{g}_\xi
	$$  
is an ideal of $\mathfrak{g}$.
\end{lemma}


\begin{proof}
Since $\dim\left(\mathfrak{g}_\xi\right)=1$ for all $\xi\in\Delta$, we have $[\mathfrak{g}_\xi,\mathfrak{g}_\xi]=0$ and $[\mathfrak{n},\mathfrak{g}_\xi]=0$ in view of the fact that $\ad_{\mathfrak{g}}(n)$ is a nilpotent operator over $\mathfrak{g}_{\xi}$ for all $n\in\mathfrak{n}$. 

\end{proof}


The next result will be readily employed in Section 3 in the study of  Lie $2$-algebras whenever at least two root spaces (relative to $\mathfrak{t}$) have different dimensions. 


\begin{proposition}
\label{proposicion51}
Let $\xi,\eta\in\Delta$ be two $\mathfrak{t}$-roots. If $\mathfrak{g}_{\xi}^{[2]}\not\subseteq\ker(\eta)$, then $\dim\left(\mathfrak{g}_{\eta}\right)=\dim\left(\mathfrak{g}_{\xi+\eta}\right)$.
\end{proposition}


\begin{proof}
From Remark~\ref{rem:nilp} we can assume that $\xi\neq\eta$. Consider $x\in\mathfrak{g}_{\xi}$ such that $x^{[2]}\notin \ker(\eta)$ and the restricted adjoint mapping $\ad_{\mathfrak{g}}(x)\big|_{\mathfrak{g}_{\eta}}\colon\mathfrak{g}_{\eta}\longrightarrow\mathfrak{g}_{\xi+\eta}$. Set $x^{[2]}=t+n$ for some $t\in\mathfrak{t}$ and $n\in\mathfrak{n}$.

	Let $y\in\mathfrak{g}_{\eta}$ be such that $\ad_{\mathfrak{g}}(x)(y)=0$. This implies that $\left[t+n,y\right]=\left[x^{[2]},y\right]=0$, hence $\eta(t)y=\ad_{\mathfrak{g}}(n)(y)=[n,y]=0$. Since $x^{[2]}\notin \ker(\eta)$ and $0\neq\eta(t+n)=\eta(t)$ we have $y=0$. Therefore $\ad_{\mathfrak{g}}(x)\big|_{\mathfrak{g}_{\eta}}\colon\mathfrak{g}_{\eta}\longrightarrow\mathfrak{g}_{\xi+\eta}$ is injective.

Analogously, we can prove that $\ad_{\mathfrak{g}}(x)\big|_{\mathfrak{g}_{\xi+\eta}}\colon\mathfrak{g}_{\xi+\eta}\longrightarrow \mathfrak{g}_{\eta}$ is an injective linear map, since $(\xi+\eta)(t)=\xi(t)+\eta(t)=0+\eta(t)=\eta(t)\neq0$.

\end{proof}







\begin{lemma}
\label{lem:dimneq}
Let $\mathfrak{g}$ be a Lie $2$-algebra with toral rank $r$ and $\{\alpha_1,\alpha_2,\dots,\alpha_k\}\subseteq\Delta$ be linearly independent $\mathfrak{t}$-roots. If $\dim(\mathfrak{g}_{\alpha_i})\neq\dim(\mathfrak{g}_{\alpha_1+\alpha_i})$ for all $2\leq i\leq k$, then $\mathfrak{g}_{\alpha_1}^{[2]}\subseteq I\oplus \mathfrak{n}$, where $I\subseteq \mathfrak{t}$ and $\dim I\leq r-k$.
\end{lemma}


\begin{proof}
Recall that $\mathfrak{g}_{\alpha_1}^{[2]}\subseteq \ker(\alpha_1)$. Let us consider a basis $\{t_1,t_2,\dots,t_r\}$ of $\mathfrak{t}$ such that its dual basis of $\mathfrak{t}^*$ satisfies $\alpha_i(t_j)=\delta_{i,j}$ for all $1\leq i,j\leq k$ and $\delta_{i,j}$ is the Kronecker delta. From the latter proposition it is perceived 
	$$
	\mathfrak{g}_{\alpha_1}^{[2]}\subseteq \bigcap_{1\leq i\leq k} \ker(\alpha_i).
	$$ 
Since $t_i\not\in\ker(\alpha_i)$, then $\mathfrak{g}_{\alpha_1}^{[2]}\subseteq \vspan\{t_{k+1},\dots,t_r\}\oplus \mathfrak{n}$ where $\vspan\{t_{k+1},\dots,t_r\}=0$ if $k=r$.

\end{proof}


\subsection{Root space decomposition of a Lie 2-algebra of toral rank 3}

Henceforth we will fix a centerless Lie $2$-algebra $\mathfrak{g}$ (not necessarily simple Lie algebra) with ${MT}(\mathfrak{g})=3$. In this subsection we describe the root space decomposition of $\mathfrak{g}$.

Consider $\left\{t_1,t_2,t_3\right\}$ a fixed basis of toral elements of $\mathfrak{t}$. By dint of Remark \ref{patrondepatrones} we can identify any $\lambda\in\mathfrak{t}^{*}$ as a $3$-vector of the form $\lambda=(\lambda(t_{1}),\lambda(t_{2}),\lambda(t_{3}))$, where $\lambda(t_i)\in\{0,1\}$ for all $i=1,2,3$. Therefore
	$$
	\mathfrak{t}^{*}= \{0,\alpha,\beta,\gamma,\alpha+\beta,\alpha +\gamma,\beta+\gamma,\alpha+\beta+\gamma\}
	$$
where $\alpha=(1,0,0) $, $\beta=(0,1,0)$, $\gamma= (0,0,1)$ is the canonical basis of $\mathfrak{t^{*}}$. Since $\mathfrak{h} =\mathfrak{t}\oplus \mathfrak{n}$, the root space decomposition of $\g$ relative to $\mathfrak{t}$ is
	\begin{equation}
	\label{patron}
	\mathfrak{g}=\mathfrak{t} \oplus \mathfrak{n} \oplus \mathfrak{g}_{\alpha}\oplus\mathfrak{g}_{\beta} \oplus\mathfrak{g}_{\gamma}\oplus\mathfrak{g}_{\alpha + \beta}\oplus\mathfrak{g}_{\alpha+ \gamma}\oplus \mathfrak{g}_{\beta+\gamma}\oplus\mathfrak{g}_{\alpha + \beta + \gamma}.
	\end{equation}
	
	It follows from Proposition \ref{wefef} that there exist at least three linearly independent $\mathfrak{t}$-roots, which implies that $\dim\left(\g\right)\geq 6$. To study $\mathfrak{g}$ we use the cardinality of $\Delta$ (denoted by $\text{Card}(\Delta)$). Then, without loss of generality we may assume that $\mathfrak{g}_{\alpha}$, $\mathfrak{g}_{\beta}$, and $\mathfrak{g}_{\gamma}$ are nonzero which will result in the following possible cases:
 
\begin{enumerate}[1.]
\item If $\text{Card}(\Delta)=3$, we have $\Delta_1=\{\alpha,\beta,\gamma\}$.
\item If $\text{Card}(\Delta)=4$, we have: 
	\begin{enumerate}[a.]
	\item $\Delta_{2}=\{\alpha,\beta,\gamma,\alpha+\beta\}$,
	\item $\Delta_{3}=\{\alpha,\beta,\gamma,\alpha+\beta+\gamma\}$.
	\end{enumerate}
\item If $\text{Card}(\Delta)=5$, we have:
	\begin{enumerate}[a.]
	\item $\Delta_{4}=\{\alpha,\beta,\gamma,\alpha+\beta,\alpha+\gamma\}$,
	\item $\Delta_{5}=\{\alpha,\beta,\gamma,\alpha+\beta,\alpha+\beta+\gamma\}$. 
	\end{enumerate}
\item If $\text{Card}(\Delta)=6$, we have:
	\begin{enumerate}[a.]
	\item $\Delta_{6}=\{\alpha,\beta,\gamma,\alpha+\beta,\alpha+\gamma,\beta+\gamma\}$, 
	\item $\Delta_{7}=\{\alpha,\beta,\gamma,\alpha+\beta,\alpha+\gamma,\alpha+\beta+\gamma\}$. 
\end{enumerate}
\item If $\text{Card}(\Delta)=7,$ then $\Delta_0=\mathfrak{t}^*\backslash\{0\}$.
\end{enumerate}

For $0\leq i\leq 7,$ we denote by $\g^{\Delta_{i}}$ the Lie $2$-algebra with its associated root space $\Delta_{i}$ from the previous list. In other words,
	\begin{equation*}
	\mathfrak{g}^{\Delta_i}=\mathfrak{t} \oplus\mathfrak{n} 	\oplus\underset{\xi\in\Delta_i}\bigoplus\mathfrak{g}_\xi.
	\end{equation*}

We conclude this section with the following technical lemma which provides a basis of the ideal $I$ in Lemma~\ref{lem:dimneq} in terms of the $\mathfrak{t}$-roots $\{\alpha,\beta,\gamma\}$.


\begin{lemma}
\label{lem:tech}
Set $\mathfrak{g}=\mathfrak{g}^{\Delta_i}$ with $0\leq i\leq 7$ and $t_{\alpha}=t_1$, $t_{\beta}=t_2$, $t_{\gamma}=t_3$. Let $\xi\in\Delta_i$ be a $\mathfrak{t}$-root, and $\rho$ and $\omega$ different $\mathfrak{t}$-roots in $\{\alpha,\beta,\gamma\}$. 
\begin{enumerate}[(i)]
\item If $\xi\in\left\{\alpha,\beta,\gamma\right\}$, then
	$$
	\left[\mathfrak{g}_{\xi},\mathfrak{g}_{\xi}\right]\subseteq\vspan\left\{t_{\sigma}\mid \sigma+\xi\in\Delta_i\right\}\oplus\mathfrak{n}.	
	$$
\item If $\xi=\rho+\omega$, then
	$$
	\left[\mathfrak{g}_{\xi},\mathfrak{g}_{\xi}\right]\subseteq\vspan\left\{t_{\rho}+t_{\omega}\right\}\oplus\vspan\left\{t_{\alpha+\beta+\gamma+\xi}\mid \alpha+\beta+\gamma\in\Delta_i\right\}\oplus\mathfrak{n}.
	$$
\item If $\xi=\alpha+\beta+\gamma$, then
	$$
	\left[\mathfrak{g}_{\xi},\mathfrak{g}_{\xi}\right]\subseteq\vspan\left\{t_{\alpha}+t_{\gamma}, t_{\beta}+t_{\gamma}\right\}\oplus\mathfrak{n}.
	$$
\end{enumerate}

\end{lemma}


\begin{proof}
Let $\xi\in\Delta_i$ be a $\mathfrak{t}$-root and let $x=[x_{\xi},y_{\xi}]\in \left[\mathfrak{g}_{\xi},\mathfrak{g}_{\xi}\right]\subseteq\mathfrak{h}=\mathfrak{t}\oplus\mathfrak{n}$ be an arbitrary element. Hence $x=t+n$ for some $t\in\mathfrak{t}$ and $n\in\mathfrak{n}$. Clearly,
	$$
	t=\alpha(t)t_{\alpha}+\beta(t)t_{\beta}+\gamma(t)t_{\gamma}.
	$$
From Remark~\ref{rem:nilp} one has $[\mathfrak{g}_\xi,\mathfrak{g}_\xi]\subseteq\ker\xi$ which yields $\xi(t)=0$. If $\rho\in\Delta_i$ satisfies that $\rho+\xi\not\in\Delta_i$, for any $z\in\mathfrak{g}_{\rho}$ from the Jacobi identity we have 
	\begin{align*}
   0&= [z,x]+[y_{\xi},[z,x_{\xi}]]+[x_{\xi},[y_{\xi},z]]=[z,x]\\
	&= [z,t]+[z,n]=\rho(t)z+\ad_{\g}(n)(z).
	\end{align*} 
Since $\ad_{\mathfrak{g}}(n)$ is a nilpotent operator over $\mathfrak{g}_{\rho}$, we can conclude that $\rho(t)=0$. 

By the above yields the following:

\begin{enumerate}
\item[\textit{(i)}] If $\xi\in\left\{\alpha,\beta,\gamma\right\}$, then $t\in\vspan\left\{t_{\sigma}\mid \sigma+\xi\in\Delta_i\right\}$.
\item[\textit{(ii)}] If $\xi=\rho+\omega$ with $\rho,\omega\in\left\{\alpha,\beta,\gamma\right\}$, then $\rho(t)+\omega(t)=0$, which implies that $\rho(t)=\omega(t)$ and therefore 
	$$
	t\in\vspan\left\{t_{\rho}+t_{\omega}\right\}\oplus\vspan\left\{t_{\alpha+\beta+\gamma+\xi}\mid \alpha+\beta+\gamma\in\Delta_i\right\}.
	$$
\item[\textit{(iii)}] If $\xi=\alpha+\beta+\gamma$, then $\alpha(t)+\beta(t)+\gamma(t)=0$.
\end{enumerate}

\end{proof}


\section{On simplicity of Lie 2-algebras}
\label{sec:centerless}

The goal of this section is to give a necessary condition for simple Lie $2$-algebras which have a triangulable Cartan subalgebra, with $MT(\mathfrak{g})=3$. 

Let us start with $\Delta = \Delta_{0}$. To this end, we study the $\mathfrak{t}$-root space decomposition of a centerless Lie $2$-algebra $\mathfrak{g}$ of toral rank $3$, that is,
	$$	
	\mathfrak{g}=\mathfrak{t}\oplus\mathfrak{n}\oplus \mathfrak{g}_{\alpha}\oplus\mathfrak{g}_{\beta}\oplus \mathfrak{g}_{\gamma}\oplus \mathfrak{g}_{\alpha+\beta}\oplus \mathfrak{g}_{\alpha+\gamma}\oplus \mathfrak{g}_{\beta+\gamma}\oplus \mathfrak{g}_{\alpha+\beta+\gamma}.
	$$

Note that, if $\Gamma\subseteq\mathfrak{t}^*$ is an additive subgroup, then
	$$
	I=\sum_{\xi\not\in\Gamma}\mathfrak{g}_{\xi}+\sum_{\xi,\eta\not\in\Gamma}[\mathfrak{g}_{\xi},\mathfrak{g}_{\eta}]	
	$$
is an ideal of $\mathfrak{g}$. It is a proper ideal provided that $\sum_{\xi\not\in\Gamma}[\mathfrak{g}_{\xi},\mathfrak{g}_{\xi}]\neq\mathfrak{h}$.

Denote by $\mathbb{Z}\Delta$ the subgroup of $\mathfrak{t}^*$ spanned by the roots.







\begin{lemma}
Let $\mathfrak{g}$ be a centerless Lie $2$-algebra with toral rank $n$ and $\Gamma$ the finite group spanned by $n-1$ linearly independent roots all of them belonging to $\Delta$. Suppose that there exists $\alpha \in \Gamma$ such that $\dim(\mathfrak{g}_{\xi})\neq\dim(\mathfrak{g}_{\alpha})$ for each $\xi\in\mathbb{Z}\Delta\backslash\Gamma$. Then $\mathfrak{g}$ contains a nonzero proper ideal $I$.
\end{lemma}


\begin{proof}
Define $I$ as shown above. If $\xi\not\in\Gamma$, then $\xi+\alpha\not\in\Gamma$ as well, whence
$\dim\mathfrak{g}_{\alpha}\neq\dim\mathfrak{g}_{\xi+\alpha}$, and therefore $\mathfrak{g}_{\xi}^{[2]}\subseteq\ker(\alpha)$. It follows that $[\mathfrak{g}_{\xi},\mathfrak{g}_{\xi}]\subseteq\ker(\alpha)$ for each $\xi\not\in\Gamma$, and so $I\neq\mathfrak{g}$.
\end{proof}


\begin{corollary}
Let $\mathfrak{g}$ be a simple Lie $2$-algebra of toral rank $n$. If there exists $\alpha\in\Delta$ such that $\dim(\mathfrak{g}_{\alpha})=d$, then there exist more than $n-1$ linearly independent roots with the same dimension $d$. 
\end{corollary}


\begin{corollary}
\label{Simple3}
Let $\mathfrak{g}$ be a simple Lie $2$-algebra of toral rank $3$. There exists $d$ such that for any root $\alpha\in\Delta$, we have $\dim (\mathfrak{g}_{\alpha})=d$.
\end{corollary}


\begin{proof}
Suppose that there exist $\alpha,\ \beta\in\Delta$ such that $d_1=\dim (\mathfrak{g}_{\alpha})\neq\dim (\mathfrak{g}_{\alpha+\beta})=d_2$. Thus, the latter Corollary split $\Delta$ as
\[
S_1=\{\alpha,\beta,\gamma\},\qquad\qquad S_2=\{\alpha+\beta,\alpha+\gamma,\beta+\gamma,\alpha+\beta+\gamma\}.
\]
such that for any $\xi\in S_i$, we have $\dim (\mathfrak{g}_{\xi})=d_i$. Set $\xi=\alpha,\beta$. Since, $\dim (\mathfrak{g}_{\alpha+\beta})\neq\dim (\mathfrak{g}_{\xi})$ and $\dim (\mathfrak{g}_{\xi})\neq\dim (\mathfrak{g}_{\alpha+\beta+\gamma})$,  hence
\begin{align*}
\mathfrak{g}_{\xi}^{[2]}&\subset \ker (\alpha)\cap \ker (\beta),\\
\mathfrak{g}_{\beta+\gamma}^{[2]}&\subset \ker (\alpha)\cap \ker (\beta+\gamma)\subset\ker(\alpha+\beta+\gamma),\\
\mathfrak{g}_{\alpha+\gamma}^{[2]}&\subset \ker (\alpha+\gamma)\cap \ker (\beta)\subset\ker(\alpha+\beta+\gamma).
\end{align*}
Now, consider $\xi=\alpha,\gamma$. From, $\dim (\mathfrak{g}_{\alpha+\gamma})\neq\dim (\mathfrak{g}_{\xi})$ and $\dim (\mathfrak{g}_{\gamma})\neq\dim (\mathfrak{g}_{\alpha+\beta+\gamma})$ follow
	\begin{align*}
	\mathfrak{g}_{\xi}^{[2]}&\subset \ker (\alpha)\cap \ker (\gamma),\\
	\mathfrak{g}_{\alpha+\beta}^{[2]}&\subset \ker (\alpha+\beta)\cap \ker (\gamma)\subset\ker(\alpha+\beta+\gamma).
\end{align*}
Analogously with $\xi=\beta,\gamma$. Thus, we conclude that for $\xi=\alpha,\beta,\gamma$,
	\[
	\mathfrak{g}_{\xi}^{[2]}\subset\ker (\alpha)\cap \ker(\beta)\cap \ker (\gamma)\subset \ker(\alpha+\beta+\gamma),
	\]
therefore, for any $\xi\in\Delta$, we obtain
\[
[\mathfrak{g}_{\xi},\mathfrak{g}_{\xi}]\subset \mathfrak{g}_{\xi}^{[2]}\subset\ker (\alpha+\beta+\gamma).
\]
Hence, $\mathfrak{g}$ is not simple, which is an absurd. Hence, all root spaces have the same dimension.

\end{proof}


\subsection{Simple Lie 2-algebras of toral rank 3}
\label{sec:simple}

	The main goal of this subsection is to study simple Lie $2$-algebras with a triangulable Cartan subalgebra, of toral rank 3 and suggest a way to classify them. We start with the non-existence of such algebras if there are less than 7 $\mathfrak{t}$-roots in the decomposition~\eqref{patron}, i.e. $\mathfrak{g}^{\Delta_i}$ is not simple for $1\leq i\leq 7$. 

	Recall that Proposition~\ref{wefef} guarantees that the dimension of simple Lie $2$-algebras of toral rank $3$ has a lower bound of $6$. The apex of this section is the proof of a necessary condition for simplicity of Lie $2$-algebras (see Theorem~\ref{Payarin}). As an immediate corollary to this latter theorem we shall increase the lower bound of the dimension to $16$. In order to achieve all this, we first ought to prove the following theorem which is also interesting in its own right.


\begin{theorem}
\label{payarin}
Let $1\leq i\leq 7$. There are not simple Lie $2$-algebras   $\mathfrak{g}^{\Delta_i}$ with a triangulable Cartan subalgebra.
\end{theorem}


\begin{proof}
Suppose for a contradiction that $\mathfrak{g}=\mathfrak{g}^{\Delta_i}$ is a simple Lie $2$-algebra of toral rank $3$. By simplicity of $\mathfrak{g}$, we have $\mathfrak{h}=\sum_{\xi\in\Delta_i}[\mathfrak{g}_\xi,\mathfrak{g}_\xi]$. Consider $\rho,\omega, \sigma$ different $\mathfrak{t}$-roots in $\{\alpha,\beta,\gamma\}$. In the following cases we use the notation and the  inclusions from Lemma~\ref{lem:tech} without making any explicit mention.

\textbf{Case 1:} Since $[\mathfrak{g}_{\xi},\mathfrak{g}_{\xi}]\subseteq\mathfrak{n}$ for all $\xi\in\Delta_1$, we have $\mathfrak{t}=0$. 

\textbf{Case 2.a:} A straightforward computation shows that $\mathfrak{t}\subseteq \vspan\left\{t_{\alpha},t_{\beta}\right\}$. 

\textbf{Case 2.b:} In this case we get $[\mathfrak{g}_{\xi},\mathfrak{g}_{\xi}]\subseteq\mathfrak{n}$ for all $\xi\in\Delta_3\backslash\left\{\alpha+\beta+\gamma\right\}$ and $[\mathfrak{g}_{\alpha+\beta+\gamma},\mathfrak{g}_{\alpha+\beta+\gamma}]\subseteq\mathfrak{n}\oplus\vspan\{t_{\alpha}+t_{\gamma},t_{\beta}+t_{\gamma}\}$. Since $\alpha,\beta\in\Delta_3$ and 
	$$
	\beta+\gamma=\alpha+(\alpha+\beta+\gamma),\ \ \alpha+\gamma=\beta+(\alpha+\beta+\gamma)\not\in\Delta_3
	$$ 
we conclude that $\alpha=\beta=0$ over $[\mathfrak{g}_{\alpha+\beta+\gamma},\mathfrak{g}_{\alpha+\beta+\gamma}]$. Hence, $\mathfrak{t}=0$.

\textbf{Case 3.a:} Note that $[\mathfrak{g}_{\alpha},\mathfrak{g}_{\alpha}]\subseteq \mathfrak{n}\oplus\vspan\{t_{\beta},t_{\gamma}\}$. For any $\omega\neq\alpha$ it follows that $[\mathfrak{g}_{\omega},\mathfrak{g}_{\omega}]\subseteq\mathfrak{n}\oplus\vspan\{t_{\alpha}\}$ and $[\mathfrak{g}_{\alpha+\omega},\mathfrak{g}_{\alpha+\omega}]\subseteq\mathfrak{n}\oplus\vspan\{t_{\alpha}+t_{\omega}\}$. However, $\alpha=0$ over $[\mathfrak{g}_{\omega},\mathfrak{g}_{\omega}]$ and $[\mathfrak{g}_{\alpha+\omega},\mathfrak{g}_{\alpha+\omega}]$, because $\alpha+\rho\in\Delta_4$ and 
	$$
	\alpha+\beta+\gamma=\omega+(\alpha+\rho),\ \ \rho+\omega=\alpha+\omega+(\alpha+\rho)\not\in\Delta_4.
	$$ 
This implies that $\mathfrak{t}\subseteq \vspan\{t_{\beta},t_{\gamma}\}$.

\textbf{Case 3.b:} For $\omega,\rho\in\{\alpha,\beta\}$ we have $[\mathfrak{g}_{\omega},\mathfrak{g}_{\omega}]\subseteq\mathfrak{n}$, because $[\mathfrak{g}_{\omega},\mathfrak{g}_{\omega}]\subseteq\mathfrak{n}\oplus\vspan\{t_{\rho}\}$, $\alpha+\beta+\gamma\in\Delta_5$ and $\rho+\gamma\not\in\Delta_5$. In addition, $[\mathfrak{g}_{\alpha+\beta},\mathfrak{g}_{\alpha+\beta}]\subseteq\mathfrak{n}\oplus\vspan\{t_{\alpha}+t_{\beta},t_{\gamma}\}$ and $[\mathfrak{g}_{\alpha+\beta+\gamma},\mathfrak{g}_{\alpha+\beta+\gamma}]\subseteq\mathfrak{n}\oplus\vspan\{t_{\alpha}+t_{\gamma},t_{\beta}+t_{\gamma}\}$. However, $\alpha,\beta\in\Delta_5$ but $\beta+\gamma,\alpha+\gamma\not\in\Delta_5$, which implies that $[\mathfrak{g}_{\alpha+\beta+\gamma},\mathfrak{g}_{\alpha+\beta+\gamma}]\subseteq\mathfrak{n}$. We conclude that $\mathfrak{t}\subseteq \vspan\{t_{\alpha}+t_{\beta},t_{\gamma}\}$.

\textbf{Case 4.a:} Note that $[\mathfrak{g}_{\omega},\mathfrak{g}_{\omega}]\subseteq\mathfrak{n}\oplus\vspan\{t_{\rho},t_{\sigma}\}$. However, $\rho+\sigma\in\Delta_6$ but $\alpha+\beta+\gamma\not\in\Delta_6$. This means, similarly to the previous cases, that $[\mathfrak{g}_{\omega},\mathfrak{g}_{\omega}]\subseteq\mathfrak{n}\oplus\vspan\{t_{\rho}+t_{\sigma}\}$, since $\rho+\sigma=0$ over $[\mathfrak{g}_{\omega},\mathfrak{g}_{\omega}]$. In addition, $[\mathfrak{g}_{\omega+\rho},\mathfrak{g}_{\omega+\rho}]\subseteq\mathfrak{n}\oplus\vspan\{t_{\omega}+t_{\rho}\}$. Hence, $\mathfrak{t}\subseteq \vspan\{t_{\alpha}+t_{\beta},t_{\alpha}+t_{\gamma},t_{\beta}+t_{\gamma}\}=\vspan\{t_{\alpha}+t_{\gamma},t_{\beta}+t_{\gamma}\}$.

\textbf{Case 4.b:} In the same manner as in the previous cases, we see that $[\mathfrak{g}_{\alpha},\mathfrak{g}_{\alpha}]\subseteq\mathfrak{n}\oplus\vspan\{t_{\beta}+t_{\gamma}\}$, because $\alpha+\beta+\gamma\in\Delta_7$ and $\beta+\gamma\not\in\Delta_7$. For $\omega\neq\alpha$, $[\mathfrak{g}_{\omega},\mathfrak{g}_{\omega}]\subseteq\mathfrak{n}\oplus\vspan\{t_{\alpha}\}$ and $[\mathfrak{g}_{\alpha+\omega},\mathfrak{g}_{\alpha+\omega}]\subseteq\mathfrak{n}\oplus\vspan\{t_{\alpha}+t_{\beta}+t_{\gamma}\}$, because $\alpha+\rho\in\Delta_7$ and $\omega+\rho\not\in\Delta_7$ for $\rho\not\in\{\alpha,\omega\}$. Besides, $[\g_{\alpha+\beta+\gamma},\g_{\alpha+\beta+\gamma}]\subseteq\mathfrak{n}\oplus\vspan\{t_{\beta}+t_{\gamma}\}$, since $\alpha\in\Delta_7$ and $\beta+\gamma\not\in\Delta_7$. We conclude that $\mathfrak{t}\subseteq \vspan\{t_{\alpha},t_{\beta}+t_{\gamma},t_{\alpha}+t_{\beta}+t_{\gamma}\}=\vspan\{t_{\alpha},t_{\beta}+t_{\gamma}\}$.

\vspace{.2cm}

In any case, we have $\mathfrak{t}\subseteq\mathfrak{t}_0$ for some subspace $\mathfrak{t}_0$ of dimension 2. This contradicts the hypothesis of the toral rank of $\mathfrak{g}$.

\end{proof}


The next theorem is the most important result of this paper and gives a necessary condition for a Lie 2-algebra  to be simple:


\begin{theorem}
\label{Payarin}
Let $\mathfrak{g}$ be a simple Lie $2$-algebra with a triangulable Cartan subalgebra, of toral rank $3$. Then, $\dim(\mathfrak{g}_{\xi})=\dim(\mathfrak{g}_{\sigma})$ for any non-zero $\xi,\sigma\in\mathfrak{t}^*$.
\end{theorem}


\begin{proof}
Theorem~\ref{payarin} guarantees $\Delta=\Delta_0$. Furthermore, Colorally~\ref{Simple3} implies that all dimension of the $\mathfrak{t}$-root spaces are equals.

\end{proof}

 
As an immediate consequence of the latter theorem and Lemma \ref{lem:dim=1} we have the following:


\begin{corollary}
There are not simple Lie $2$-algebras $\mathfrak{g}$ with a triangulable Cartan subalgebra, of toral rank $3$ with $\dim(\g)\leq 16$.
\end{corollary}

\medskip

\textbf{Acknowledgements}. 
The second author thanks to the Professor A. Grishkov for his valuable guidance. It is extended to the Universidade Federal do Amazonas for its hospitality, as well as the support of Universidad Tecnológica de Bolívar during his postdoctoral stay. This research was partially supported by the Coordena\c{c}\~{a}o de Aperfei\c coamento de Pessoal de N\'ivel Superior -- Brasil (CAPES) --  Finance Code 001.


\bibliographystyle{plain}
\bibliography{biblio}

\end{document}